\documentclass{amsart}
\usepackage{graphicx}
\usepackage[mathscr]{eucal}
\usepackage{amssymb}
\newtheorem{thm}{Theorem}[section]
\newtheorem{cor}[thm]{Corollary}
\newtheorem{lem}[thm]{Lemma}
\newtheorem{prop}[thm]{Proposition}
\theoremstyle{definition}

\theoremstyle{remark}

\newtheorem{ex}[thm]{Example}
\newcommand{\norm}[1]{\left\Vert#1\right\Vert}

\newcommand{\set}[1]{\left\{#1\right\}}

\newcommand{\Pm}{\mathrm{Prim}}
\newcommand{\mset}{\emptyset}

\begin{document}
\baselineskip=18pt
\title{COMPACT SUBSETS OF THE GLIMM SPACE OF A $C^*$-ALGEBRA}
\author{Aldo J. Lazar}
\address{School of Mathematical Sciences\\
         Tel Aviv University\\
         Tel Aviv 69978, Israel}
\email{aldo@post.tau.ac.il}

\thanks{}%
\subjclass{46L05}
\keywords{primitive ideal space, complete regularization}
\date{March 15, 2012}

\begin{abstract}

If $A$ is a $\sigma$-unital $C^*$-algebra and $a$ is a strictly positive element of $A$ then for every compact subset of the complete
regularization $\mathrm{Glimm}(A)$ of $\Pm(A)$ there exists $\alpha > 0$ such that $K\subset \set{G\in \mathrm{Glimm}(A) \mid \norm{a + G}\geq
\alpha}$. This extends the conclusion of \cite[Theorem 2.5]{Da} to all $\sigma$-unital $C^*$-algebras. However, there is a $C^*$-algebra $A$ and a
compact subset of $\mathrm{Glimm}(A)$ that is not included in any set of the form $\set{G\in \mathrm{Glimm}(A) \mid \norm{a + G}\geq \alpha}$.
\end{abstract}

\maketitle
\section{Introduction}

The lack of good separation properties on the primitive ideal space of a $C^*$-algebra is a serious obstacle in obtaining useful non-commutative
versions of the Gelfand-Naimark theorem for the commutative algebras. One way to circumvent this impediment is to pass to the complete
regularization of the primitive ideal space. A method to obtain the complete regularization of $\Pm(A)$ for a $C^*$-algebra $A$ was exposed in
\cite[III, \S]{DH} and we shall briefly review it here. Two primitive ideals $P_1$ and $P_2$ are equivalent, $P_1 \approx P_2$, if $f(P_1) =
f(P_2)$ for every bounded continuous real function $f$ on $\Pm(A)$. Each equivalence class is a hull-kernel closed subset of $\Pm(A)$ and one
associates to it its kernel. The family of ideals thus obtained was denoted in \cite[p. 351]{AS} by $\mathrm{Glimm}(A)$ and the quotient space
$\Pm(A)/\approx$ was naturally identified with it. The quotient map $q_A$ takes a primitive ideal $P$ to the kernel of its equivalence class.
Two topologies that can be put on $\mathrm{Glimm}(A)$ are of interest to us. One is the quotient topology $\tau_q$, the other is the weakest
topology for which the functions on $\mathrm{Glimm}(A)$ defined by dropping to this space the bounded real continuous functions on $\Pm(A)$ are
continuous. The latter topology is completely regular and is denoted $\tau_{cr}$. Obviously $\tau_q\geq \tau_{cr}$ and cases when equality
occurs or does not were discussed in \cite{AS} and \cite{L}. Ways to represent $A$ as an algebra of continuous fields with base space
$(\mathrm{Glimm}(A),\tau_{cr})$ were discussed in \cite{DH} and \cite{AS}. Other uses of $\mathrm{Glimm}(A)$ can be found in \cite{EW}.

The continuous fields considered have to vanish at infinity and the one point compactification $\mathrm{Glimm}(A)\cup \{\infty\}$ considered in
\cite{DH} was defined by using the complements in $\mathrm{Glimm}(A)\cup \{\infty\}$ of the $\tau_{cr}$-compact sets $\set{G\in
\mathrm{Glimm}(A) \mid \norm{a + G}\geq \alpha}$, $a\in A$ and $\alpha > 0$, as basic neighbourhoods of the point at infinity. Of course, one
can use  the family of all the $\tau_{cr}$-compact subsets of $\mathrm{Glimm}(A)$ to get the finest one point compactification. The stated
objective of \cite{Da} was to show that if $A$ is quasicentral that is, no primitive ideal of $A$ contains the center of $A$, then each
$\tau_{cr}$-compact subset of $\mathrm{Glimm}(A)$ is included in a set of the form $\set{G\in \mathrm{Glimm}(A) \mid \norm{a + G}\geq \alpha}$
for some $a\in A$ and $\alpha
> o$ meaning that the two compactifications are the same, see Theorem 2.5 there. In \cite[p. 351]{AS} this result was improved by showing that $a$ can be taken in the
center of $A$.

Our aim is to show that one can cover the $\tau_{cr}$-compact subsets of the Glimm space by sets determined by norm inequalities as in
\cite[Theorem 2.5]{Da} in other two situations: when the $C^*$-algebra has a countable approximate identity or when the quotient map on the
Glimm space is open. We also show that there are situations when the $\tau_{cr}$-compact sets cannot be covered in this manner.

We shall use the well known fact that a $C^*$-algebra with a countable approximate identity (often called a $\sigma$-unital $C^*$-algebra) has a
strictly positive element $a$ that is, an element such that $aAa$ is dense in $A$. We shall freely make use of the following two equalities for
$x$ an element of the $C^*$-algebra $A$ and $\alpha > 0$:
\begin{multline*}
 q_A(\set{P\in \Pm(A) \mid \norm{x + P} > \alpha}) = \set{G\in \mathrm{Glimm}(A) \mid \norm{x + G} > \alpha},\\
 q_A(\set{P\in \Pm(A) \mid \norm{x + P}\geq \alpha}) = \set{G\in \mathrm{Glimm}(A) \mid \norm{x + G}\geq \alpha}.
\end{multline*}
They follow immediately from these facts:
\begin{enumerate}
   \item $q_A(P)\subset P$ for every $P\in \Pm(A)$
   \item for each $G\in \mathrm{Glimm}(A)$ and $x\in A$ there is $P\in \mathrm{hull}(G)$ such that $\norm{x + P} = \norm{x + G}$.
\end{enumerate}

\section{Results}

\begin{thm} \label{T:T}

   Let $A$ be a $C^*$-algebra with a countable approximate identity and $a\in A$ a strictly positive element. Then for every $\tau_{rc}$-compact subset
   $K$ of $\mathrm{Glimm}(A)$ there exists $\alpha > 0$ such that $K\subset \set{G\in \mathrm{Glimm}(A) \mid \norm{a + G}\geq \alpha}$.

\end{thm}

\begin{proof}

   The existence of a countable approximate identity implies $\tau_q = \tau_{cr}$ by \cite[Theorem 2.6]{L} so we shall work with $\tau_q$ in what
   follows.

   For every natural number $n$ let $O_n := \set{P\in \Pm(A) \mid \norm{a + P} > 1/n}$, $C_n := \set{P\in \Pm(A) \mid \norm {a + P} \geq 1/n}$ and
   $K_n := q_A(C_n)$ where $q_A : \Pm(A)\to \mathrm{Glimm}(A)$ is the quotient map. Then $O_n$ is open in $\Pm(A)$ and $C_n$ is compact by
   \cite[Propositions 3.3.2 and 3.3.7]{D}. Thus $K_n$ is a compact hence closed subset of the Hausdorff space $\mathrm{Glimm}(A)$. We have $\Pm(A)
   = \cup_{n=1}^{\infty} O_n = \cup_{n=1}^{\infty} C_n$ since $a$ is strictly positive, $\mathrm{Glimm}(A) = \cup_{n=1}^{\infty} K_n$, and
   $K_n\subset K_{n+1}$ for every $n$.

   A subset $F$ of $\mathrm{Glimm}(A)$ is closed if and only if each $F\cap K_n$ is closed. It is clear that if $F\subset \mathrm{Glimm}(A)$ is
   closed than $F\cap K_n$ is closed. Suppose now that $F\cap K_n$ is closed for every $n$. We want to show that $q_A^{-1}(F)$ is closed. To this
   end, let $P_0\in \overline{q_A^{-1}(F)}$. Then $P_0\in O_m$ and $q_A(P_0)\in K_m$ for some $m$. Let $U$ be an arbitrary open neighbourhood of
   $q_A(P_0)$ in $\mathrm{Glimm}(A)$. Thus $q_A^{-1}(U)\cap O_m$ is an open neighbourhood of $P_0$ hence there exists $P\in q_A^{-1}(F)\cap
   q_A^{-1}(U)\cap O_m$. We get $q_A(P)\in F\cap U\cap K_m$. We showed that $U\cap F\cap K_m\neq \mset$ and we can conclude that $q_A(P_0)\in
   \overline{F\cap K_m} = F\cap K_m$. But then $P_0\in q_A^{-1}(F\cap K_m)\subset q_A^{-1}(F)$. It follows that $\overline{q_A^{-1}(F)} =
   q_A^{-1}(F)$.

   Now we claim that if $K\subset \mathrm{Glimm}(A)$ is $\tau_q$-compact then there exists $m$ such that $K\subset K_m = \set{G\in \mathrm{Glimm}(A)
   \mid \norm{a + G}\geq 1/m}$. If not then there exists $G_n\in K\setminus K_n$ for every $n$. We shall show that the infinite set $F := \set{G_n
   \mid n\in \mathbb{N}}$ is both closed and discrete. Being a subset of the compact set $K$ this a contradiction and by this our claim will be
   established. Let $F_1$ be any subset of $F$. Since the sequence $\set{K_n}$ is non-decreasing, $F_1\cap K_n$ is finite for every $n$. By the
   previous paragraph $F_1$ is closed and we are done.

\end{proof}

For a quasicentral $C^*$-algebra with a countable approximate identity one can improve a little the strengthened version of \cite[Theorem
2.5]{Da} outlined in \cite[p. 351]{AS}.

\begin{cor}

   Let $A$ be a quasicentral $C^*$-algebra with an approximate identity. Then its center, $Z(A)$, contains a strictly positive element $z$ of $A$ so
   for every $\tau_{cr}$-compact subset $K$ of $\mathrm{Glimm}(A)$ there exists $\alpha > 0$ such that
   \[
   K\subset \set{G\in \mathrm{Glimm}(A) \mid \norm{z + G}\geq \alpha}.
   \]

\end{cor}

\begin{proof}

   As observed in \cite[p. 351]{AS}, $\mathrm{Glimm}(A)$ is homeomorphic to the maximal ideal space of $Z(A)$. Thus the latter space is
   $\sigma$-compact hence $Z(A)$ contains a strictly positive element $z$ for $Z(A)$. But no positive linear functional of $A$ can be trivial on
   $Z(A)$ since $A$ is quasicentral. Thus $z$ is strictly positive for $A$ and the conclusion follows from Theorem \ref{T:T}.

\end{proof}

It is possible to describe precisely when a $\tau_{cr}$-compact subset of the Glimm space of a $C^*$-algebra can be covered by a compact set
determined by a norm inequality: it must be the image by the quotient map of a compact subset of the primitive ideal space.

\begin{prop} \label{P:cover}

   Let $A$ be a $C^*$-algebra and $K$ a $\tau_{cr}$-compact subset of $\mathrm{Glimm}(A)$. There are $a\in A$ and $\alpha > 0$ such that $K\subset
   \set{G\in \mathrm{Glimm}(A) \mid \norm{a + G}\geq \alpha}$ if and only if there exists a compact subset $C\subset \Pm(A)$ such that $q_A(C) = K$.
   A $\tau_{cr}$-compact that satisfies the above conditions is also $\tau_q$-compact.

\end{prop}

\begin{proof}

   Suppose the compact subset $C$ of $\Pm(A)$ satisfies $q_A(C) = K$. For every $P_0\in C$ let $x_0$ be a positive element in $A\setminus P_0$.
   Then $P_0\in \set{P\in \Pm(A) \mid \norm{x_0 + P} > \norm{x_0 + P_0}/2}$ and we get in this way an open cover of $C$. By compactness we get
   positive elements $\set{x_i}_{i=1}^n$ of $A$ and positive scalars $\set{\alpha_i}_{i=1}^n$ such that $C\subset \cup_{i=1}^n \set{P\in \Pm(A)
   \mid \norm{x_i + P} > \alpha_i}$. With $x := \sum_{i=1}^n x_i$ and $\alpha := \min \set{\alpha_i \mid 1\leq i\leq n}$ we have
   \[
    C\subset \cup_{i=1}^n \set{P\in \Pm(A) \mid \norm{x_i + P} > \alpha_i}\subset \set{P\in \Pm(A) \mid \norm{x + P} > \alpha}.
   \]
   Then $K = q_A(C)\subset q_A(\set{P\in \Pm(A) \mid \norm{x + P} > \alpha}) = \set{G\in \mathrm{Glimm}(A) \mid \norm{x + G} > \alpha}$.
   Thus $K\subset \set{G\in \mathrm{Glimm}(A) \mid \norm{x + G}\geq \alpha}$

   Suppose now that $K\subset \set{G\in \mathrm{Glimm}(A) \mid \norm{x + G}\geq \alpha}$ for some $x\in A$ and $\alpha > o$. Then $C :=
   q_A^{-1}(K)\cap \set{P\in \Pm(A) \mid \norm{x + P}\geq \alpha}$ is compact being a relatively closed subset of the compact set $\set{P\in
   \Pm(A) \mid \norm{x + P}\geq \alpha}$ and $q_A(C) = K$ since $q_a(\set{P\in \Pm(A) \mid \norm{x + P}\geq \alpha}) = \set{G\in
   \mathrm{Glimm}(A) \mid \norm{x + G}\geq \alpha}$. Obviously, $K = q_A(C)$ is $\tau_q$-compact.

\end{proof}

To obtain another situation in which the conclusion of \cite[Theorem 2.5]{Da} holds we need the following lemma.

\begin{lem} \label{L:open}

   Let the Hausdorff space $Y$ be a quotient of the locally compact space $X$. If the quotient map, $q$, is open then for every compact subset
   $K$ of $Y$ there exists a compact subset $C$ of $X$ such that $q(C) = K$.

\end{lem}

\begin{proof}

   For each $x\in q^{-1}(K)$ let $E_x$ be a compact subset of $X$ such that $x\in \mathrm{int}(E_x)$. Then
   \[
    K\subset \cup_{x\in q^{-1}(K)} q(\mathrm{int}(E_x)).
   \]
   By passing to a finite subcover of this open cover of $K$ we get $\set{x_i}_{i=1}^n\subset q^{-1}(K)$ such that $K\subset \cup_{i=1}^n
   q(\mathrm{int}(E_{x_i}))\subset \cup_{i=1}^n q(E_{x_i})$. Now $D := \cup_{i=1}^n E_{x_I}$ is a compact subset of $X$ and $C := q^{-1}(K)\cap D$
   is a relatively closed subset of it hence compact too. Clearly $q(C) = K$.

\end{proof}

If $q_A$ is $\tau_q$-open then $\tau_q = \tau_{cr}$ by \cite[p. 351]{AS}; if $q_A$ is $\tau_{cr}$-open then obviously it is also $\tau_q$-open.
The following proposition follows immediately from Proposition \ref{P:cover} and Lemma \ref{L:open}.

\begin{prop} \label{P:open}

   Let $A$ be a $C^*$-algebra for which the quotient map $q_A ; \Pm(A)\to \mathrm{Glimm}(A)$ is open. Then for every $\tau_{cr}$-compact
   $K\subset \mathrm{Glimm}(A)$ there exist $a\in A$ and $\alpha > 0$ such that $K\subset \set{G\in \mathrm{Glimm}(A) \mid \norm{a + G}\geq \alpha}$.

\end{prop}

Remark that by \cite[Theorem 3.3]{AS} Proposition \ref{P:open} applies to the class of quasi-standard $C^*$-algebras.

In view of the above results it is natural to ask if the conclusion of \cite[Theorem 2.5]{Da} always holds when $A$ ia a $C^*$-algebra for which
$\tau_{cr} = \tau_q$ on $\mathrm{Glimm}(A)$.

\section{examples}

We know now three classes of $C^*$-algebras for which the $\tau_{cr}$-compact subsets of their Glimm spaces satisfy the conclusion of
\cite[Theorem 2.5]{Da}: the quasicentral $C^*$-algebras, the $C^*$-algebras that have a countable approximate identity, and those for which the
quotient map of the primitive ideal space onto the Glimm space is open. In this section we shall show that none of these classes contains one of
the other.

In the following $\omega$ will denote the first infinite ordinal and $\Omega$ will stand for the first uncountable ordinal. All the spaces of
ordinals will be considered with their order topology.

First we shall give an example of a quasicentral algebra $A$ that does not have a countable approximate identity and for which the map $q_A$ is
not open which will be based upon \cite[Example 4.8]{A}. The idea to use it for the stated purpose comes from \cite[p. 352]{AS}.

\begin{ex}

   Let $H$ be an infinite dimensional Hilbert space, $\mathcal{L}(H)$ the algebra of the bounded operators on $H$, and $\mathcal{K}(H)$ its ideal
   of compact operators. We denote by $A_1$ the unital $C^*$-algebra of all the sequences $x = \set{x_n}_{n=1}^{\infty}$ with $x_n\in
   \mathcal{L}(H)$ such that $\lim_{n\to \infty} \norm{x_{2n} - \lambda(x)\mathbf{1}_H} = 0$ for some scalar $\lambda(x)$ and $\lim_{n\to \infty}
   \norm{x_{2n+1} - (\lambda(x)\mathbf{1}_H + c(x))} = 0$ for some compact operator $c(x)$. The primitive ideals of $A_1$ are $P_n := \set{x\in A_1
   \mid x_n = 0}$, $Q_n := \set{x\in A_1 \mid x_n\in \mathcal{K}(H)}$, $J := \set{x\in A_1 \mid \lambda(x) = 0}$, and $P_{\infty} := \set{x\in A_1
   \mid \lambda(x) = 0, c(x) = 0}$. We have $\overline{\set{P_n}} = \set{P_n, Q_n}$, $\overline{\set{P_{\infty}}} = \set{P_{\infty} , J}$ in
   $\Pm(A_1)$ and $ Z_n := \set{P_n, Q_n}$, $Z_{\infty} := \set{P_{\infty}, J}$ are equivalence classes for the relation $\approx$. The subset
   $$
    \set{P_{2n} \mid n\in \mathbb{N}}\cup \set{Q_n \mid n\in \mathbb{N}}\cup \set{J}
   $$
   is closed in $\Pm(A_1)$ hence its complement $E : \set{P_{2n+1} \mid n\in \mathbb{N}}\cup \set{P_{\infty}}$ is open. $\mathrm{Glimm}(A_1)$
   consists of the converging sequence $\set{Z_n}_{n\geq 1}$ together with its limit $Z_{\infty}$. The image of $E$ by $q_{A_1}$ is $F :=
   \set{Z_{2n+1} \mid n\in \mathbb{N}}\cup Z_{\infty}$. Let now $A := \mathcal{C}_0([0,\Omega),A_1)$. Then $\Pm(A)$ can be identified with
   $[0,\Omega)\times \Pm(A_1)$, $\mathrm{Glimm}(A) = [0,\Omega)\times \mathrm{Glimm}(A_1)$. The open subset $[0,\Omega)\times E$ of $\Pm(A)$ is
   mapped by the quotient map $q_A$ onto $[0,\Omega)\times F$ which is not open in $\mathrm{Glimm}(A)$. The non $\sigma$-compact space $[0,\Omega)$
   is homeomorphic with the closed subset $[0,\Omega)\times \set{Q_1}$ of $\Pm(A)$. Thus $\Pm(A)$ is not $\sigma$-compact and $A$ does not have a
   countable approximate identity. For each ordinal $\alpha < \Omega$ define $f_{\alpha}\in A$ by $f_{\alpha}(\beta) = \mathbf{1}_{A_1}$ for $0\leq
   \beta\leq \alpha$ and $f_{\alpha}(\beta) = 0$ otherwise. We have $f_{\alpha}\in Z(A)$ for every $\alpha$ and no primitive ideal of $A$ contains
   all the functions $f_{\alpha}$. Thus $A$ is a quasicentral $C^*$-algebra.

\end{ex}

\begin{ex}

   We shall exhibit now a $\sigma$-unital non quasicentral $C^*$-algebra $A$ for which the map $q_A$ is not open.
   Let now $H$ be an infinite dimensional separable Hilbert space, $\mathcal{K}(H)$ the algebra of all compact operators on $H$ and $D$ the
   subalgebra of $\mathcal{K}(H)$ consisting of all the diagonal operators with respect to a fixed orthonormal basis of $H$. In \cite[Example
   9.2]{DH} one considers the $C^*$-subalgebra $A$ of $\mathcal{C}([-1,1],\mathcal{K}(H))$ of all the functions $f$ for which $f(t)\in D$, $f(t) =
   \mathrm{diag}(f_{ii}(t))$ say, whenever $0\leq t\leq 1$. Being a separable $C^*$-algebra $A$ has a countable approximate identity. The center of
   $A$ is $\set{f\in A \mid f(t) = 0 \; \mathrm{for} -1\leq t < 0}$. Obviously it is contained in the primitive ideal $P := \set{f\in A \mid f(-1)
   = 0}$ thus $A$ is not quasicentral. The Glimm space of $A$ is not $\tau_{cr}$-compact, see \cite[p. 127]{DH}, so it follows from \cite[Theorem
   2.1]{AS} that $q_A$ is not open.

\end{ex}

\begin{ex}

   Now we are going to discuss an example of a non $\sigma$-unital non quasicentral $C^*$-algebra $A$ for which $q_A$ is open. As above let
   $\mathcal{K}(H)$ be the $C^*$-algebra of all the compact operators on an infinite dimensional Hilbert space. Then $A :=
   \mathcal{C}_0([0,\Omega), \mathcal{K}(H))$ is not $\sigma$-unital since $\Pm(A)$ is homeomorphic to $[0,\O
   mega)$ and this space is not
   $\sigma$-compact. The center of $A$ is trivial thus $A$ cannot be quasicentral. Finally, $q_A$ is just the identity map on $\Pm(A)$.

\end{ex}

Lastly we shall show that there exist a $C^*$-algebra $A$ and a $\tau_{cr}$-compact subset K of $\mathrm{Glimm}(A)$ that is not
$\tau_q$-compact. By Proposition \ref{P:cover} for such $K$ there do not exist $a\in A$ and $\alpha > 0$ such that $K\subset \set{G\in
\mathrm{Glimm}(A) \mid \norm{a + G}\geq \alpha}$. The $C^*$-algebra which we treat is a particular case of a class of $C^*$-algebras constructed
by D. W. B. Somerset in the Appendix of \cite{L} for which $\tau_{cr}\neq \tau_q$.

\begin{ex}

   Let $Y := [0,\Omega)\times [0,\omega)$, $S := \{\Omega\}\times [0,\omega)$, and $T := [0,\Omega)\times \{\omega\}$. Let $y$ be an element not
   in $Y$ and consider the space $X := Y\cup \{y\}$ with the topology determined by the following requirements: $Y$ is embedded homeomorphically
   into $X$, $\{y\}$ is an open subset whose closure is $S\cup \{y\}$. Then $X$ is a locally compact $T_0$-space and, following \cite[p.
   155]{L}, we shall produce a $C^*$-algebra $A$ whose primitive ideal space is homeomorphic to $X$.

   For $B := \mathcal{C}_0(Y)$ and $D := \mathcal{C}_0(S)$ let $\pi_1 : B\to D$ be the restriction map. Let $\{p_n\}_{n=1}^{\infty}$ be a
   sequence of infinite dimensional mutually orthogonal projections on the infinite dimensional separable Hilbert space $H$. We define an embedding
   $\rho : D\to \mathcal{L}(H)$ by $\rho(f) := \sum_{n=1} f(\Omega,n)p_n$. Remark that $\rho(D)\cap \mathcal{K}(H) = \{0\}$. Set
   $E := \rho(D) + \mathcal{K}(H)$ and let $\pi_2 : E\to D$ be the
   natural quotient map. We denote by $A$ the pullback of $\pi_1$ and $\pi_2$ that is $A := \set{(b,e)\in B\oplus E \mid \pi_1(b) = \pi_2(e)}$.
   It is easily seen that $\Pm(A)$ is homeomorphic to $X$.

   The quotient map $q_A$ maps every point in $Y\setminus S$ to itself and $S\cup \{y\}$ ia $\approx$-class which we shall denote by $z$. Thus
   the Glimm space of $A$ can be identified with $\mathcal{G} := (Y\setminus S)\cup \{z\}$. We claim that $K := T\cup \{z\}$ is a $\tau_{cr}$-compact
   subset of $\mathcal{G}$ that it is not $\tau_q$-compact. The $\tau_{cr}$-compactness of $K$ follows immediately once we prove that every
   $\tau_{cr}$-neighbourhood of $z$ contains a set of the form $((\alpha,\Omega)\times \{\omega\})\cup \{z\}$ for some ordinal $\alpha$ since
   $[0,\alpha]\times \{\omega\}$ is clearly compact. A basic $\tau_{cr}$-neighbourhood $\mathcal{U}$ of $z$ is given by a bounded real valued continuous
   function $g$ on $\mathcal{G}$ such that $g(z) = 1$: $\mathcal{U} = \set{t\in \mathcal{G} \mid |g(t) - 1| < 1}$. The continuity of $g$ on
   $([0,\Omega)\times \{n\})\cup z$ for $0\leq n < \omega$ entails the existence of an ordinal $\alpha_n$, $0\leq \alpha_n < \Omega$, such that
   $g(\beta) = 1$ whenever $\alpha_n < \beta < \Omega$. Set $\alpha := \sup_n \alpha_n$; then $0\leq \alpha < \Omega$ and
   $g(\beta,n) = 1$ for $\alpha < \beta < \Omega$ and $0\leq n < \omega$. It follows that $g(\beta,\omega) = 1$ if $\alpha < \beta < \Omega$ thus
   $((\alpha,\Omega)\times \{\omega\})\subset \mathcal{U}$ as needed. The subset $T = q_A(T)$ of $K$ is $\tau_q$-closed since $q_A^{-1}(T) = T$
   and $T$ is closed in $X$. Therefore $K$ cannot be $\tau_q$-compact since $T$ is not compact.

\end{ex}

\bibliographystyle{amsplain}
\bibliography{}

\end{document}